\documentclass{article}

\usepackage{authblk}
\usepackage{geometry}
\usepackage[titletoc]{appendix}
\usepackage{enumitem}
\usepackage{booktabs,diagbox}
\usepackage{xcolor}

\definecolor{darkgreen}{rgb}{0.0, 0.6, 0.0}

\usepackage{tikz} 
\usetikzlibrary{cd}
\usetikzlibrary{arrows}
\usepackage{subcaption}
\usepackage{graphicx}

\usepackage{amssymb,amsmath,amsthm,mathtools,slashed,bm} 
\usepackage[cal=boondoxo]{mathalfa}
\def\-{\raisebox{.75pt}{-}}

\usepackage[alphabetic,initials]{amsrefs}

\numberwithin{table}{section}
\numberwithin{equation}{section}

\theoremstyle{definition}
\newtheorem{defn}{Definition}[section]

\newtheorem{exmp}{Example}[section]
\newtheorem{rmk}{Remark}[section]
\theoremstyle{plain}
\newtheorem{lem}{Lemma}[section]
\newtheorem{prop}{Proposition}[section]
\newtheorem{thm}{Theorem}[section]
\newtheorem{cor}{Corollary}[section]

\usepackage{hyperref,cleveref}

\title{A coadjoint orbit-like construction for Jordan superalgebras}
\author[1]{Florio M. Ciaglia}
\author[2]{Shuhan Jiang}
\author[2]{Jürgen Jost}
\author[3]{Lorenz Schwachhöfer}
\date{}
\affil[1]{Department of Mathematics, University Carlos III de Madrid, Legan\'es, Madrid, Spain}
\affil[2]{Max Planck Institute for Mathematics in the Sciences, Leipzig, Germany}
\affil[3]{Faculty for Mathematics, TU Dortmund University, Dortmund, Germany}

\begin{document}
	
	\maketitle
	\begin{abstract}
		We investigate the canonical pseudo-Riemannian metrics associated with Jordan-analogues of the coadjoint orbits for pseudo-Euclidean Jordan superalgebras.
	\end{abstract}

	\section{Introduction}

    Both Lie and Jordan algebras are  algebras that play an important role in mathematical formulations of quantum mechanics, but their defining identities and their structure theories look quite different from each
    other. Nevertheless, an analogy emerges when (under some technical assumptions about integrability that we suppress in this introduction) we consider the orbits of the structure group \cite{Ciaglia23}.  The structure group of a finite-dimensional (non-associative) algebra $A$ is the Lie subgroup of the general linear group $\mathrm{GL}(A)$ generated by left multiplications $L_a$ with $a \in A$. The group $\mathrm{GL}(A)$  acts canonically on the dual $A^*$ of $A$.  In the case where $A = \mathfrak{g}$ is the Lie algebra of the connected Lie group $G$, the orbits of the structure group in $\mathfrak{g}^*$ are known as the coadjoint orbits of the group. Kirillov's pioneering study reveals that there exists a correspondence between coadjoint orbits and the unitary representations of $G$ \cites{Kirillov1962,Kirillov2012}.
	The coadjoint orbits carry a symplectic form, defined as
	\begin{align}\label{kirisymp}
		\omega_{\xi}(X,Y) := \xi([X,Y]), \quad \xi \in \mathfrak{g}^*, \quad X,Y \in \mathfrak{g}.
	\end{align}
	The study of the geometric quantization of these symplectic manifolds provides an alternative point of view on Kirillov's orbit method.
	
	It has been recently shown in \cite{Ciaglia23} that for an arbitrary algebra $A$, replacing the Lie bracket by the product of $A$, the formula \eqref{kirisymp} defines a $(0,2)$-tensor ${\mathcal g}$ on the regular part of the orbits of the structure group in the dual space $A^*$. In particular, if $A = J$ is a Jordan algebra, then ${\mathcal g}$ is symmetric, and a description of regular points in $J^\ast$ was given. In case of an Euclidean Jordan algebra $J$, where $J \cong J^\ast$ canonically, the orbit through the identity element consists of regular points only, and ${\mathcal g}$ defines a Riemannian metric on this orbit. As it turns out, this Riemannian metric is closely related to the Fisher-Rao metric when $J=\mathbb{R}^{n}$, and to the Bures-Helstrom metric tensor when $J =\mathcal{B}_{sa}(\mathcal{H})$, the Jordan algebra of self-adjoint linear operators on a finite-dimensional Hilbert space $\mathcal{H}$. These two metric tensors are, respectively, at the heart of the geometrical formulation of classical and quantum estimation theory \cite{C-J-S-2020-02}, and their role in classical and quantum information geometry is difficult to overestimate \cite{A-N-2000,Paris-2009}.
	
	Both Lie and Jordan algebras admit natural super extensions. In fact, when we pass to the super setting, the   analogy between the two types of algebras becomes clearer.     The Jacobi identity for the left translations of a Lie superalgebra is
    \begin{align}\label{jaclie}
        	[L_a, L_{[b,c]}]=[L_{[a,b]},L_c] + (-1)^{|a||b|}[L_b, L_{[a,c]}],
    \end{align}
    where $|a|$ denotes the parity of $a$ ($=0$ for even, $=1$ for odd elements), while for a Jordan superalgebra we have
    \begin{align}\label{jacjor}
			[L_a, L_{\{b,c\}}]=[L_{\{a,b\}},L_c] - (-1)^{|a||b|}[L_b, L_{\{a,c\}}].
	\end{align}
	where $\{\cdot,\cdot\}$ is the algebraic product of $J$ and $[\cdot,\cdot]$ is the Lie superbracket of its structure Lie algebra $\mathfrak{g}(J)$. The analogy between \eqref{jaclie} and \eqref{jacjor} points the way for a structure theory.
            
    A nice classification result of simple Lie and Jordan superalgebra over an algebraically closed field of characteristic 0 was obtained by Kac in the late 70s \cite{Kac77}. Two decades later, another classification result for simple Jordan superalgebras with semi-simple even parts over fields of characteristic different from $2$ was obtained by Racine and Zel’manov \cite{Racine2003}. The structure Lie superalgebras of simple Jordan superalgebras were computed by Barbier and Coulembier in \cite{Barbier2018}. Meanwhile, Kirillov's orbit method was extended to nilpotent Lie supergroups by Salmasian \cite{Salmasian2010}. This paper will provide a Jordan counterpart of the story, which can be viewed as a super extension of the results of \cite{Ciaglia23}.
          
	We are mainly interested in the pseudo-Euclidean Jordan superalgebras, namely, Jordan superalgebras which can be equipped with an even, symmetric, nondegenerate, associative bilinear form $\beta$. In this context, we also detect a new phenomenon distinct from non-super cases: the positive definiteness of the even part of $\beta$ does not ensure the positive definiteness of the even part of the canonical Killing form of Jordan superalgebras, as illustrated by a class of Jordan superalgebras $D(t)$. 
	
	\section{Preliminaries}
	
	\subsection{Superalgebras}
	
	A superalgebra $A$ is a $\mathbb{Z}_2$-graded real algebra $A=A_0 \oplus A_1$ such that $A_iA_j \subset A_{i+j}$. An element $a$ in $A_i$ is said to have parity $|a|=i$. $a$ is called even if $|a|=0$ and odd if $|a|=1$. $A$ is said to be commutative if
	$
		ab=(-1)^{|a||b|}ba
	$
	and anticommutative if
	$
	ab=-(-1)^{|a||b|}ba.
	$
	A Lie superalgebra $L$ is an anticommutative superalgebra satisfying the super Jacobi identity
	\begin{align*}
		[a, [b,c]]=[[a,b],c] + (-1)^{|a||b|}[b, [a,c]].
	\end{align*} 
    	
	Let $A$ be a $\mathbb{Z}_2$-graded real vector space with $\dim(A_0)=m$ and $\dim(A_1)=n$. We use the notation $m|n$ to denote the (super) dimension of $A$. The endomorphism algebra $\mathrm{gl}(A)$ of $A$ is an associative superalgebra under compositions of endomorphisms. Moreover, $\mathrm{gl}(A)$ can be given a Lie superalgebra structure by setting its Lie superbracket to be
	$
		[f,g]=f \circ g - (-1)^{|f||g|}g \circ f.
	$
    In the case of $A$ being a superalgebra, an endomorphism $d$ is said to be a (super)derivation if 
    \begin{align}\label{supder}
    	d(ab)=d(a)b+(-1)^{|d||a|}ad(b).
    \end{align}
    Derivations of $A$ form a Lie subsuperalgebra of $\mathrm{gl}(A)$, denoted by $\mathrm{Der}(A)$. Let $L_a$ denote the endomorphism on $A$ induced by the left multiplication of $a \in A$ on $A$. $L_a$ generate another Lie subsuperalgebra of $\mathrm{gl}(A)$, denoted by $\mathfrak{g}(A)$. We call $\mathfrak{g}(A)$ the structure Lie superalgebra of $A$. Let $\mathrm{Der}_0(A) = \mathrm{Der}(A) \cap \mathfrak{g}(A)$. $\mathrm{Der}_0(A)$ is again a Lie superalgebra. The elements in $\mathrm{Der}_0(A)$ are called inner derivations of $A$. Clearly, $\mathrm{Der}_0(A)$ is an ideal of $\mathrm{Der}(A)$. Derivations which are not inner are called outer derivations.
    
    For $l \in \mathrm{gl}(A)$, one can decompose it into four parts $l_{00}, l_{01}, l_{10}, l_{11}$, where $l_{ij}$ maps $A_i$ into $A_j$. The supertrace of $l$ is defined as
    \begin{align*}
    	\mathrm{str}(l):=\mathrm{tr}(l_{00})-\mathrm{tr}(l_{11}).
    \end{align*}
    \begin{lem}\label{str}
    	$\mathrm{str}([f,g])=0$ for $f,g \in \mathrm{gl}(A)$.
    \end{lem}
    \begin{proof}
    	$\mathrm{str}([f,g])=\mathrm{tr}([f_{00},g_{00}])-\mathrm{tr}([f_{11},g_{11}])=0$.
    \end{proof}

    A bilinear form $\beta$ on $A$ is said to be symmetric if $\beta(a,b)=(-1)^{|a||b|}\beta(b,a)$; non-degenerate if $\beta(a,b)=0$ for all $b$ implies that $a=0$; associative if $\beta(ab,c)=\beta(a,bc)$.  If $A$ is a commutative superalgebra, one can define a canonical even symmetric bilinear form $\tau$ on $A$ by
    \begin{align*}
    	\tau(a,b)=\mathrm{str}(L_{ab}).
    \end{align*}
    We call a commutative superalgebra $A$ pseudo-Euclidean if it has an even, symmetric, non-degenerate, and associative bilinear form $\beta$, and semi-simple if the canonical bilinear form $\tau$ of $A$ is non-degenerate and associative. Let $(A,\beta)$ be a pseudo-Euclidean commutative superalgebra. An ideal $I$ of $A$ is said to be $\beta$-non-degenerate if the restriction of $\beta$ to $I$ is non-degenerate. A $\beta$-non-degenerate ideal $I$ of $A$ is said to be $\beta$-irreducible if it has no $\beta$-non-degenerate ideal other than $\{0\}$ and itself.
    
    \begin{prop}\label{betadecomp}
    	Every finite-dimensional pseudo-Euclidean commutative superalgebra $(A,\beta)$ admits a decomposition $A=\bigoplus_i I_i$ such that
    	\begin{enumerate}
    		\item $I_i$ is $\beta$-irreducible for all $i$;
    		\item $\beta(I_i, I_j)=0$ whenever $i \neq j$.
    	\end{enumerate}
    \end{prop}
    \begin{proof}
    	Let $I$ be a $\beta$-non-degenerate ideal of $A$ and $I^{\perp}$ be the orthogonal complement of $I$ respect to $\beta$. Since $\beta$ is non-degenerate, $I \cap I^{\perp} = \{0\}$ and we have $J = I \oplus I^{\perp}$. For $x \in I^{\perp}$, $y \in I$, and $z \in A$, we have $\beta(xz,y)=\beta(x,zy)=0$ since $yz \in I$. Therefore, $I^{\perp}$ is also a $\beta$-non-degenerate ideal of $A$. The proposition can be then proved by induction because the dimension of $A$ is finite.
    \end{proof}

    \begin{rmk}
    	Let $R$ be an associative commutative superalgebra. An $R$-superalgebra $A$ is defined to be a superalgebra $A$ equipped with a left $R$-module structure such that $R_iA_j \subset A_{i+j}$ and
    	\begin{align*}
    		r(ab)=(ra)b=(-1)^{|r||a|}a(rb)
    	\end{align*}
    	for $r \in R$ and $a,b \in A$. If $R$ is unital, we require additionally that $1 a = a$. All of the above discussions can be easily extended to the case of $R$-superalgebras by simply requiring every multilinear map $f$ to be $R$-multilinear. That is,
    	\begin{align*}
    		rf(a_1,\cdots,a_q)=(-1)^{|r|(|f|+|a_1|+\cdots+|a_{p-1}|)}f(a_1,\cdots,ra_p,\cdots,a_q).
    	\end{align*} 
    \end{rmk}

    Let $(A,\beta)$ be a pseudo-Euclidean superalgebra of dimension $m|n$. Since $\beta$ is non-degenerate, $n$ must be even. The odd part $\beta_1=\beta|_{A_1}$ of such $\beta$ can always be brought into the form
    \begin{equation*}
    	\beta_1 = 
    	\begin{bmatrix}
    		0 & I_{n/2} \\
    		-I_{n/2} & 0 
    	\end{bmatrix}
    \end{equation*}
    where $I_{n/2}$ is the $n/2 \times n/2$ identity matrix. The even part $\beta_0=\beta|_{A_0}$ of $\beta$ is a symmetric matrix and can be diagonalized into the form
    \begin{equation*}
    	\beta_0 = 
    	\begin{bmatrix}
    		I_r & 0 \\
    		0 & -I_s 
    	\end{bmatrix}
    \end{equation*}
    where $r+s=m$. $(r,s)$ is said to be the signature of $(A,\beta)$. Note that the signature depends on the choice of $\beta$, hence, it is not an intrinsic invariant of $A$. $(A,\beta)$ is said to be Euclidean if it has signature $(m,0)$, i.e., if $\beta_0$ is positive definite. $A$ is said to be positive if the canonical bilinear form $\tau$ of $A$ is associative and $\tau_0$ is positive definite.
    
    \subsection{Supermanifolds}
    
    A superdomain $U^{m|n}$ is a ringed space $(U, \mathcal{O}_{U^{m|n}})$ where $U$ is an open subset of $\mathbb{R}^m$ and $\mathcal{O}_{U^{m|n}}(V) = C^{\infty}(V)\otimes \Lambda(\mathbb{R}^n)$ for each open subset $V \subset U$, where $\Lambda(\mathbb{R}^n)$ is the exterior algebra of $\mathbb{R}^n$.
    A supermanifold is a ringed space $\mathcal{M}=(M,\mathcal{O}_{\mathcal{M}})$ where $M$ is a second countable Hausdorff topological space and $\mathcal{O}_{\mathcal{M}}$ is a sheaf of (unital associative) commutative superalgebras over $M$ such that for each $x \in M$ there exist an open neighborhood $U \ni x$ and a ringed space isomorphism
    \begin{align*}
    	(U,\mathcal{O}_{\mathcal{M}}|_U) \xrightarrow{\varphi_U=(\widetilde{\varphi_U},\varphi_U^{\sharp})} U^{m|n},
    \end{align*}
    where $U^{m|n}$ is a superdomain \cite{Leites80}. $(U,\varphi_U)$ is called a chart of $\mathcal{M}$. 
    
    By definition, the underlying topological space $M$ of a supermanifold $\mathcal{M}$ is a smooth manifold. $M$ itself can also be viewed as a supermanifold with the structure sheaf being the sheaf $C^{\infty}$ of smooth functions. There exists a unique ringed space morphism $M \xrightarrow{(\mathrm{Id}, \epsilon)} \mathcal{M}$, where $\epsilon: \mathcal{O}_{\mathcal{M}}(V) \rightarrow C^{\infty}(V)$ is known as the body map for any $V$ open in $M$. A morphism between two supermanifolds is just a morphism between the corresponding ringed spaces.
    
    An open neighborhood $U \in x$ is called splitting if $\mathcal{O}_M(U) \cong C^{\infty}(U) \otimes \Lambda(\mathbb{R}^n)$. A (local) coordinate system of $\mathcal{M}$ is the data $(U,u^i,\theta^j)$ where $U$ is a splitting open neighborhood in $M$, $u^1, \cdots, u^m \in \mathcal{O}_M(U)_0$ such that $(U, \epsilon(u^i))$ is a coordinate system of $M$, and $\theta^1, \cdots, \theta^n \in \mathcal{O}_M(U)_1$ freely generate $\mathcal{O}_M(U)$ over $C^{\infty}(U)$. Every coordinate system determines a chart of $\mathcal{M}$. 
    
    A vector field $X$ over $\mathcal{M}$ is a derivation of the commutative superalgebra $\mathcal{O}_{\mathcal{M}}(M)$. The tangent bundle $T\mathcal{M}$ of $\mathcal{M}$ is defined as the sheaf $\mathfrak{X}_{\mathcal{M}}$ of all vector fields over $\mathcal{M}$. $\mathfrak{X}_{\mathcal{M}}$ is also called the tangent sheaf of $\mathcal{M}$. It is a locally free $\mathcal{O}_{\mathcal{M}}$-module. Indeed, let $(U,u^i,\theta^j)$ be a coordinate system, $\mathfrak{X}_{\mathcal{M}}(U)$ is a free $\mathcal{O}_{\mathcal{M}}(U)$-module with basis $(\frac{\partial}{\partial u^i}, \frac{\partial}{\partial \theta^j})$, where $\frac{\partial}{\partial u^i}$ and $\frac{\partial}{\partial \theta_j}$ act on $\mathcal{O}_{\mathcal{M}}(U)$ in the following way
    \begin{align*}
    	&\frac{\partial}{\partial u^i}(f \theta^{j_1} \cdots \theta^{j_q}) = \frac{\partial f}{\partial u^i} \theta^{j_1} \cdots \theta^{j_q}, \\
    	&\frac{\partial}{\partial \theta^j}(f \theta^j \theta^{j_1} \cdots \theta^{j_q}) =  f \theta^{j_1} \cdots \theta^{j_q}, 
    \end{align*}
    where $f \in C^{\infty}(U)$ and $j \notin\{j_1,\cdots,j_q\}$. 
    
    Let $\mathcal{O}_{\mathcal{M},x}$ denote the stalk of $\mathcal{O}_{\mathcal{M}}$ at $x \in M$. A tangent vector of $\mathcal{M}$ at $x$ is a derivation $\mathcal{O}_{\mathcal{M},x} \rightarrow \mathbb{R}$. Let $\epsilon_x: \mathcal{O}_{\mathcal{M},x} \rightarrow \mathbb{R}$ denote the algebra homomorphism given by $\mathcal{O}(M) \xrightarrow{\epsilon} C^{\infty}(M) \xrightarrow{\mathrm{ev}_x} \mathbb{R}$, where $\epsilon$ is the body map and $\mathrm{ev}_x$ is the evaluation map of a smooth function on $M$ at $x$. Each vector field over $\mathcal{M}$ determines a tangent vector at $x$ via composition with $\epsilon_x$. However, contrary to the ordinary case, a vector field is not fully determined by its tangent vectors over $M$. This is not so surprising because locally the tangent space $T_x\mathcal{M}$ of all tangent vectors at $x$ is isomorphic to $\mathrm{span}_{\mathbb{R}}\{\frac{\partial}{\partial u^i},\frac{\partial}{\partial \theta^j}\}$, while the stalk $\mathfrak{X}_{\mathcal{M},x}$ of $\mathfrak{X}_{\mathcal{M}}$ at $x$ is isomorphic to $\mathrm{span}_{\mathcal{O}_{\mathcal{M},x}}\{\frac{\partial}{\partial u^i},\frac{\partial}{\partial \theta^j}\}$. Let $\mathfrak{m}_x \subset \mathcal{O}_{\mathcal{M},x}$ denote the kernel of $\epsilon_x$. One has $\mathcal{O}_{\mathcal{M},x} \cong \mathbb{R} \oplus \mathfrak{m}_x$ and it follows that $T_x \mathcal{M} \cong \mathfrak{X}_{\mathcal{M},x}/\mathfrak{m}_x\mathfrak{X}_{\mathcal{M},x}$.
    
    A (pseudo-)Riemannian metric $\mathcal{g}$ over $\mathcal{M}$ is an even symmetric non-degenerate $\mathcal{O}_{\mathcal{M}}$-linear morphism of sheaves
    $
    	\mathfrak{X}_{\mathcal{M}} \otimes_{\mathcal{O}_{\mathcal{M}}} \mathfrak{X}_{\mathcal{M}} \rightarrow \mathcal{O}_{\mathcal{M}}.
    $
    By composing with $\epsilon_x$, $\mathcal{g}$ determines an even symmetric non-degenerate bilinear form $\mathcal{g}_x$ on the tangent space $T_x \mathcal{M}$. Likewise, $\mathcal{g}$ is not fully determined by $\mathcal{g}_x$.
    
    \subsection{Lie supergroups}
    
    A Lie supergroup $\mathcal{G}=(G,\mathcal{O}_{\mathcal{G}})$ is a group object in the category of supermanifolds. A morphism between two Lie supergroups is just the morphism between the corresponding group objects. An action of $\mathcal{G}$ on a supermanifold $\mathcal{M}$ is just the action of the corresponding group object.
    
    More concretely, a Lie supergroup $\mathcal{G}$ is uniquely determined by the so-called super Harish-Chandra pair $(G,\mathfrak{g})$ \cite{Carmeli11}, where $G$ is an ordinary Lie group, $\mathfrak{g}=\mathfrak{g}_0 \oplus \mathfrak{g}_1$ is a Lie superalgebra with $\mathfrak{g}_0 \cong \mathrm{Lie}(G)$, and $G$ acts on $\mathfrak{g}$ via a representation $\sigma$ such that $d\sigma(\xi)\xi' = [\xi,\xi']$ for $\xi \in \mathfrak{g}_0$ and $\xi' \in \mathfrak{g}$, where $d\sigma$ is the differential of $\sigma$.  Given two Lie supergroups $\mathcal{G}=(G, \mathfrak{g})$ and $\mathcal{G}'=(G',\mathfrak{g}')$. A morphism from $\mathcal{G}$ to $\mathcal{G}'$ is determined by a pair $(\phi, \varphi)$ where
    \begin{enumerate}
    	\item $\phi: G \rightarrow G'$ is a Lie group homomorphism;
    	\item $\varphi: \mathfrak{g} \rightarrow \mathfrak{g}'$ is a Lie superalgebra homomorphism;
    	\item $\phi$ and $\varphi$ are compatible in the sense that
    	\begin{align*}
    		\varphi|_{\mathfrak{g}_0} = d \phi, \quad \sigma'(\phi(g)) \circ \varphi = \varphi \circ \sigma(g), ~\forall g \in G, 
    	\end{align*}
    	where $\sigma$, $\sigma'$ are the representations of $G$, $G'$ on $\mathfrak{g}$, $\mathfrak{g}'$, respectively.
    \end{enumerate}
    
    For example, one can consider the super Harish-Chandra pair $\mathfrak{Diff}(\mathcal{M})=(\mathfrak{Diff}(\mathcal{M})_0,\mathfrak{X}_{\mathcal{M}}(M))$, where $\mathfrak{Diff}(\mathcal{M})_0$  is the infinite dimensional Lie group of automorphisms of a supermanifold $\mathcal{M}$ and $\mathfrak{X}_{\mathcal{M}}(M)$ is the Lie superalgebra of global vector fields over $\mathcal{M}$. An action of a super Lie group $\mathcal{G}=(G, \mathfrak{g})$ on $\mathcal{M}$ is just a morphism $(\phi,\varphi): \mathcal{G} \rightarrow \mathfrak{Diff}(\mathcal{M})$ of super Lie groups. $\varphi: \mathfrak{g} \rightarrow \mathfrak{X}_{\mathcal{M}}(M)$ is called the infinitesimal action of $\mathfrak{g}$.
    
    As another example,  the structure Lie superalgebra $\mathfrak{g}(A)$ of a superalgebra $A$ and the Lie group $G(A)_0 \subset \mathrm{gl}(A)_0$ generated by $\mathfrak{g}(A)_0$ form a super Harish-Chandra pair $G(A)=(G(A)_0,\mathfrak{g}(A))$ which we call as the structure Lie supergroup of $A$. $G(A)$ acts on $A$ in the obvious canonical way.
    
    There is also a well-defined notion of a $\mathcal{G}$-orbit $\mathcal{O}$ for a $\mathcal{G}$-supermanifold $\mathcal{M}$ which we will not address in details here. We simply remark that $\mathcal{O}$ is an immersed sub-supermanifold of $\mathcal{M}$ whose tangent spaces can be identified with the ``tangent spaces of a generalized distribution" over $\mathcal{M}$ induced by the infinitesimal action of $\mathfrak{g}$. 

    \section{Generalized distributions}
    
    \subsection{Generalized distributions over supermanifolds}
   
    \begin{defn}
    	Let $\mathcal{M}=(M,\mathcal{O}_{\mathcal{M}})$ be a supermanifold. A generalized distribution $\mathfrak{D}$ over $\mathcal{M}$ is a $\mathcal{O}_{\mathcal{M}}$-submodule of the tangent sheaf $\mathfrak{X}_{\mathcal{M}}$ of $\mathcal{M}$. 
    \end{defn}

    Let $\mathfrak{D}$ be a generalized distribution over $\mathcal{M}$. The quotient $T_x \mathcal{D}:=\mathfrak{D}_x/\mathfrak{m}_x \mathfrak{D}_x$ is a subspace of the super vector space $T_x \mathcal{M}$, where $\mathfrak{D}_x$ is the stalk of $\mathfrak{D}$ at $x$ and $T_x \mathcal{M}$ is the tangent space of $\mathcal{M}$ at $x$.
    
    \begin{defn}
    	The function 
    	\begin{align*}
    		\mathrm{rank}_{\mathcal{D}}: M &\rightarrow \mathbb{N} \times \mathbb{N} \\
    		x &\mapsto \dim(T_x \mathfrak{D})
    	\end{align*}
        is called the rank of $\mathcal{D}$.
    \end{defn}

    \begin{lem}
    	$\mathrm{rank}_{\mathcal{D}}$ is a lower semicontinuous function on $M$.
    \end{lem}
    \begin{proof}
    	Let $a>0$ and $x \in \mathrm{rank}_{\mathcal{D}}^{-1}((a,\infty))$. Let $\{X_i\}$ be vector fields over an open neighborhood $U$ of $x$ such that their tangent vectors at $x$ form a basis for $T_x\mathfrak{D}$. We can find another open neighborhood $V \subset U$ of $x$ such that $\{X_i\}$ are linearly independent on $V$. For any $y \in V$, $\{X_i\}$ induces a set of linearly independent vectors in $T_y \mathfrak{D}$, which implies that $\mathrm{rank}_{\mathcal{D}}(y) \geq \mathrm{rank}_{\mathcal{D}}(x)$. Hence $V \subset \mathrm{rank}_{\mathcal{D}}^{-1}((a,\infty))$.
    \end{proof}

    \begin{defn}
    	A generalized distribution $\mathcal{D}$ is called a distribution 
    	if $\mathrm{rank}_{\mathcal{D}}$ is locally constant.
    \end{defn}
    
    \begin{lem}
    	A generalized distribution $\mathcal{D}$ is a distribution if and only if it is locally free. 
    \end{lem}

    \begin{proof}
    	The ``if'' direction is trivial. To prove the other direction, let again $x \in M$ and $\{X_i\}$ be vector fields over an open neighborhood $U$ of $x$ such that their tangent vectors at $x$ form a basis for $T_x\mathfrak{D}$. We need to show that the stalks of $\{X_i\}$ generate $\mathfrak{D}_x$. If this is not the case, one can find a nonzero vector field $Y$ vanishing at $x$ and another open neighborhood $V \subset U$ of $x$ such that $\{X_i\}$ and $Y$ are linearly independent on $V$. But this cannot be because the rank of $\mathcal{D}$ at $y$ where $Y$ is non-vanishing would be then strictly lager than the rank of $\mathcal{D}$ at $x$. 
    \end{proof}

    \begin{defn}
    	Let $\mathcal{M}$ be a supermanifold and $\mathcal{D}$ be a generalized distribution over $\mathcal{M}$. Let $\phi: \mathcal{N} \rightarrow \mathcal{M}$ be an injective immersion, i.e., an immersed submanifold of $\mathcal{M}$. $\mathcal{N}$ is called an integral manifold of $\mathcal{D}$ if $d\phi_x(T_x \mathcal{N})=T_{\tilde{\phi}(x)} \mathcal{D}$ for all $x \in N$, where $d\phi_x$ is the tangent map of $\phi$ at $x$ and $\tilde{\phi}: N  \rightarrow M$ is the underlying smooth map of $\phi$. The generalized distribution $\mathcal{D}$ is called integrable if every point of $M$ is contained in a integral manifold of $\mathcal{D}$.
    \end{defn}

    \begin{defn}
    	Let $\mathcal{M}$ be a supermanifold and $\mathcal{D}$ be a generalized distribution over $\mathcal{M}$. $\mathcal{D}$ is called involutive if its sections are closed under the super Lie bracket induced from the bracket of vector fields over $\mathcal{M}$.
    \end{defn}

    In the classical case, the integrable condition and the involutive condition are equivalent for a distribution, known as the Frobenius theorem \cite{Sussmann-1973}. However, there is no longer such equivalence in the super setting because a vector field over a supermanifold is not uniquely determined by its tangent vectors. Therefore, there might be several different distributions sharing the same integral manifolds, and not all of them are involutive. However, the other direction still holds true.
    
    \begin{thm}\cite{Carmeli11}*{Theorem 6.2.1}
    	Every involutive distribution $\mathcal{D}$ over a supermanifold $\mathcal{M}$ is integrable. Moreover, for each point of $M$, there exists a unique maximal integral manifold of $\mathcal{D}$.
    \end{thm}

    \begin{rmk}
    	Luckily, we are only interested in involutive generalized distributions induced by a Lie supergroup action. Such generalized distributions are automatically integrable.
    \end{rmk}
    
    \begin{defn}
    	Let $\mathcal{D}$ be a generalized distribution over $\mathcal{M}$. Let $\mathcal{D}'\subset \mathcal{D}$ be another generalized distribution over $\mathcal{M}$. A point $p \in M$ is called $\mathcal{D}'$-regular (with respect to $\mathcal{D}$) if $\mathrm{rank}_{\mathcal{D}'}(p)=\mathrm{rank}_{\mathcal{D}}(p)$.
    \end{defn}

    \begin{cor}
    	The set of $\mathcal{D}'$-regular points is open in $M$.
    \end{cor}
     
    \begin{defn}
    	Let $\mathfrak{g}$ be a Lie superalgebra. Let $\lambda: \mathfrak{g} \rightarrow \mathfrak{X}_{\mathcal{M}}(M)$ be an infinitesimal action of $\mathfrak{g}$ on $\mathcal{M}$. The distribution $\mathfrak{D}^{\mathfrak{g}}$ defined by $\mathfrak{D}^{\mathfrak{g}}(U):= \mathrm{span}\{\lambda(X)|_U: X \in \mathfrak{g}\}$ is called the infinitesimal orbit distribution of $\mathfrak{g}$. Let $\mathfrak{m}$ be a subspace of $\mathfrak{g}$, a point $p \in M$ is called $\mathfrak{m}$-regular if it is $\mathcal{D}^{\mathfrak{m}}$-regular, where $\mathfrak{D}^{\mathfrak{m}} := \mathrm{span}\{\lambda(X): X \in \mathfrak{m}\}$.
    \end{defn}

    If the infinitesimal action of $\mathfrak{g}$ is induced by a global action of a Lie supergroup $\mathcal{G}$, then $\mathcal{D}^{\mathfrak{g}}$ is integrable and its maximal integrable manifolds are in bijection with the $\mathcal{G}$-orbits. 
    \begin{defn}
    	A $\mathcal{G}$-orbit is said to be $\mathfrak{m}$-regular if the underlying manifold of it consists only of $\mathfrak{m}$-regular points. 
    \end{defn}

    \subsection{Canonical generalized distributions on the duals of superalgebras}
    
    The dual $A^*$ of a super vector space $A=A_0\oplus A_1$ can be canonically viewed as a supermanifold $(A_0^*, \mathcal{O}_{A^*})$ where $\mathcal{O}_{A^*}$ is defined by $\mathcal{O}_{A^*}(U) = C^{\infty}(U) \otimes \Lambda(A_1)$. The body map of $A^*$ is given by the canonical projection $\epsilon: \mathcal{O}_{A^*}(U) \rightarrow C^{\infty}(U) \otimes \Lambda^0(A_1) \cong C^{\infty}(U)$. In our case, we can choose a coordinate system $(U, u_i,\theta_j)$ where $(u_i)$ is a basis of $A_0$ and $(\theta_j)$ is a basis of $A_1$. We call such $(U, u_i,\theta_j)$ a special coordinate system of $A^*$.
    
    Let $A$ be a superalgebra. Let $x_i$ be a basis of $A$ such that $x_i=u_i$ for $i=1,\dots,m$, and $x_{m+j}=\theta_j$ for $j=1,\dots,n$. The commutative superalgebra $\mathcal{O}_{A^*}(U)$ can be made into an $\mathcal{O}_{A^*}(U)$-superalgebra by considering the following product
    \begin{align}\label{dualprod}
    	\{f,g\}=\sum_{i,j} (-1)^{|x_j|(|f|+|x_i|)}(x_ix_j)(\frac{\partial f}{\partial x_i} \frac{\partial g}{\partial x_j})
    \end{align}
    for $f, g \in \mathcal{O}_{A^*}(U)$. The product inside the second round bracket of (\ref{dualprod}) should be understood as the commutative product of $\mathcal{O}_{A^*}(U)$, while the product inside the first round bracket should be understood as a product of $A$. Note that an element $a$ in $A$ can be viewed canonically as an element of $\mathcal{O}_{A^*}(U)$. The product between the two round brackets should be then understood as the commutative product of $\mathcal{O}_{A^*}(U)$.
    
    Evidently, the definition of (\ref{dualprod}) does not depend on the choice of special coordinate system of $U$. One can easily verify that the canonical injection $A \hookrightarrow \mathcal{O}_{A^*}(U)$ is an algebraic monomorphism.  Moreover, one can show the followings
    \begin{prop}
    	The $\mathcal{O}_{A^*}(U)$-superalgebra $(\mathcal{O}_{A^*}(U),\{\cdot,\cdot\})$ is commutative iff $A$ is commutative.
    \end{prop}
    \begin{proof}
    	It remains to show the "$\Leftarrow$" direction.
    	\begin{align*}
    		\{g,f\}&=\sum_{i,j} (-1)^{|x_j|(|g|+|x_i|)}(x_ix_j)(\frac{\partial g}{\partial x_i} \frac{\partial f}{\partial x_j})\\
    		&=\sum_{i,j} (-1)^{|x_i|(|g|+|x_j|)}(-1)^{(|f|+|x_i|)(|g|+|x_j|)}(x_jx_i)(\frac{\partial f}{\partial x_i} \frac{\partial g}{\partial x_j})\\
    		&=\sum_{i,j} (-1)^{|f|(|g|+|x_j|)}(-1)^{|x_i||x_j|}(x_ix_j)(\frac{\partial f}{\partial x_i} \frac{\partial g}{\partial x_j})\\
    		&=(-1)^{|f||g|}\{f,g\}.
    	\end{align*}
    \end{proof}
    
    \begin{prop}
    	$\{f,\cdot\}$ is a derivation of $\mathcal{O}_{A^*}(U)$.
    \end{prop}
    \begin{proof}
    	This follows again from direct computations.
    	\begin{align*}
    		\{f,gh\}&=\sum_{i,j}(-1)^{|x_j|(|f|+|x_i|)}(x_ix_j) (\frac{\partial f}{\partial x_i} \frac{\partial}{\partial x_j}(gh))\\
    		&=\sum_{i,j}(-1)^{|x_j|(|f|+|x_i|)}(x_ix_j)(\frac{\partial f}{\partial x_i} (\frac{\partial g}{\partial x_j}h + (-1)^{|x_j||g|}g\frac{\partial h}{\partial x_j} )) \\
    		&=\{f,g\}h+ (-1)^{|x_j||g|}(-1)^{((|f|+|x_i|)+|x_i|+|x_j|)|g|}g \{f,h\}\\
    		&=\{f,g\}h+(-1)^{|f||g|}g \{f,h\}.
    	\end{align*}
    \end{proof}
    
    $\{f,\cdot\}$ can be then viewed as a vector field over $U$, denoted by $X_f$. It is easy to check that 
    \begin{align*}
    	X_{x_i} = \sum_j (x_ix_j) \frac{\partial}{\partial x_j},
    \end{align*}
    and also that 
    \begin{align}\label{xf}
    	X_{f} = \sum_i (-1)^{|x_i|(|f|+|x_i|)}\frac{\partial f}{\partial x_i}X_{x_i}.
    \end{align}
    Let $\mathrm{Ann}(A)=\{a|ab=0~\mathrm{for}~\mathrm{all}~b \in A\}$ denote the annihilator of $A$.
    \begin{prop}\label{coreucl}
    	If $\mathrm{Ann}(A)=0$, then the kernel of $f \mapsto X_f$ consists of constant functions.
    \end{prop}
    \begin{proof}
    	By assumption, $X_{x_i} \neq 0$ for all $i$. It follows from \eqref{xf} that $X_f=0$ if and only if $\frac{\partial f}{\partial x_i}=0$ for all $i$, which is the case if and only if $f$ is a constant function.
    \end{proof} 

    Let $A^*$ be the dual of a commutative superalgebra $A$. Let $U$ be an open subset of $A_0^*$. Let $\mathfrak{H}_A(U)$ denote the submodule of $\mathfrak{X}_{A^*}(U)$ spanned by $X_{f}$ over $\mathcal{O}_{A^*}(U)$. We obtain a distribution $\mathfrak{H}_A$ over the supermanifold $A^*$, which is not involutive in general. There is a canonical even symmetric bilinear paring $\mathcal{g}$ of $\mathfrak{H}_A$, which is defined by
    \begin{align}
    	\mathcal{g}: \mathfrak{H}_A(U) \times \mathfrak{H}_A(U) \rightarrow \mathcal{O}_{A^*}(U) \label{g} \\
    	(X_f, X_g) \rightarrow \{f,g\}. \notag
    \end{align}
    and then extended bilinearly. One can easily check that $\mathcal{g}$ is well-defined and that
    $
    	\mathcal{g}(X_a, X_b) = ab.
    $
    \begin{prop}
    	$\mathcal{g}$ is non-degenerate. 
    \end{prop}
    \begin{proof}
    	By definition, $\mathcal{g}(X_f, \cdot) = 0$ if and only if $X_f(g)=0$ for all $g \in \mathcal{O}_{A^*}(U)$, $U$ open in $A_0^*$, which is the case if and only $X_f=0$.
    \end{proof}
    
    Let $\mathfrak{g}$ be a Lie superalgebra acting on $A$. The induced $\mathfrak{g}$-action on $A^*$ is given by
    \begin{align*}
    	\mathfrak{g} &\rightarrow \mathfrak{X}_{A^*}(A_0^*) \\
    	\xi &\mapsto \sum_i \xi x_i\frac{\partial}{\partial x_i}.
    \end{align*} 
    Let's take $\mathfrak{g}$ to be $\mathfrak{g}(A)$, the structure Lie superalgebra of $A$. Let $\mathfrak{m}_A$ denote the subspace of $\mathfrak{g}(A)$ spanned by $L_a$. Recall that there is the canonical injection $\iota: A \hookrightarrow \mathcal{O}_{A^*}(U)$ sending $a$ to $f_a$. Let $X_{f_a}$ denote the vector field associated to $f_{a}$. We have $X_{f_a}=\sum_j \{a,x_j\} \frac{\partial}{\partial x_j} = \sum_j(ax_j)\frac{\partial}{\partial x_j}$. It follows that $X_{L_a}=X_{f_a}$ and $\mathfrak{H}_A = \mathfrak{D}^{\mathfrak{m}_A} \subseteq \mathfrak{D}^{\mathfrak{g}(A)}$. 
    
    The infinitesimal action of $\mathfrak{g}(A)$ integrates to a global action of the structure Lie supergroup $G(A)$ on $A$. Let $\mathcal{O}$ be a $G(A)$-orbit. Let $x$ be a point in the underlying topological space of $\mathcal{O}$. Then the tangent space $T_x \mathcal{O}$ can be identified with the tangent space $T_x \mathcal{D}^{\mathfrak{g}(A)}$ of the generalized distribution $\mathcal{D}^{\mathfrak{g}(A)}$. Moreover, if $\mathcal{O}$ is $\mathfrak{m}_A$-regular, then $\mathcal{g}$ defines a pseudo-Riemannian metric over $\mathcal{O}$.
         
    \section{Jordan superalgebras}
    
	\begin{defn}
		A Jordan superalgebra $J$ is a commutative superalgebra satisfying the following super Jordan identity
		\begin{align}\label{joride}
			[L_a, L_{\{b,c\}}]=[L_{\{a,b\}},L_c] - (-1)^{|a||b|}[L_b, L_{\{a,c\}}].
		\end{align}
	    where $\{\cdot,\cdot\}$ is the algebraic product of $J$ and $[\cdot,\cdot]$ is the Lie superbracket of $\mathfrak{g}(J)$.
	\end{defn}
    \begin{rmk}
    	Choose $a$, $b$ and $c$ to be $x \in J_0$ in (\ref{joride}). One gets $[L_x,L_{\{x,x\}}]=0$. Applying this to $y \in J_0$ yields the usual Jordan identity
    	\begin{align*}
    		\{x,\{y,\{x,x\}\}\}=\{\{x,y\},\{x,x\}\}
    	\end{align*}
        for $J_0$.
    \end{rmk}
    \begin{rmk}
    	Recall that the super Jacobi identity of a Lie superalgebra $L$ is
    	\begin{align*}
    		[a, [b,c]]=[[a,b],c] + (-1)^{|a||b|}[b, [a,c]],
    	\end{align*} 
        which implies that
        \begin{align}\label{jorlieide}
        	[L_a, L_{[b,c]}]=[L_{[a,b]},L_c] + (-1)^{|a||b|}[L_b, L_{[a,c]}].
        \end{align} 
        One can then think of (\ref{joride}) as an analogue of (\ref{jorlieide}). 
    \end{rmk}
    \begin{lem}\label{kacfor1}
	    $[[L_a,L_b],L_c]=(-1)^{|b||c|}L_{[a,c,b]}$, where
	    \begin{align*}
	    	[\cdot,\cdot,\cdot]:~ &J \times J \times J \rightarrow J \\
	    	(a,b,c) &\mapsto \{a,\{b,c\}\} - \{\{a,b\},c\} 
	    \end{align*} 
        is the associator of $J$.
    \end{lem} 
    \begin{proof}
    	See \cite{Kac77}.
    \end{proof}
    \begin{prop}
    	The canonical bilinear form $\tau$ on $J$ is associative.
    \end{prop}
    \begin{proof}
    	By Lemmas \ref{str} and \ref{kacfor1},  $\tau(ab,c)-\tau(a,bc)=\mathrm{str}(L_{[a,b,c]})=\pm\mathrm{str}([[L_a,L_c],L_b])=0$.
    \end{proof}
    \begin{lem}\label{strjor}
    	$\mathfrak{g}(J)$ is spanned by $L_a$ and $[L_b,L_c]$ for $a,b,c \in J$.
    \end{lem}
    \begin{proof}
    	Using the anti-commutativity of $[\cdot,\cdot]$ and the super Jacobi identity, it is not hard to see that $\mathfrak{g}(J)$ is spanned by $L_a$ and endomorphisms of the form
    	\begin{align}
    		[[L_{a_1},L_{a_2}],L_{a_3}],\cdots],L_{a_q}],
    	\end{align}
    	where $q \geq 2$. By Lemma \ref{kacfor1}, those endomorphism can be reduced to either $L_b$ when $q$ is odd or $[L_c,L_d]$ when $q$ is even.
    \end{proof}
   
	\begin{lem}\label{unider}
		If $J$ is a unital Jordan superalgebra, then $\mathrm{Der}_0(J)=\mathrm{span}\{[L_a,L_b]|a,b \in J\}$.
	\end{lem}
    \begin{proof}
    	We first prove that $[L_a,L_b]$ is a superderivation. By (\ref{supder}), it suffices to show that $[[L_a,L_b],L_c]=L_{[L_a,L_b](c)}$. But this follows directly from Lemma \ref{kacfor1} and
        \begin{align*}
        	[a,b,c]=(-1)^{|b||c|}[L_a,L_c](b).
        \end{align*}
        If $J$ is unital, then any derivation of the identity $1 \in J$ must be $0$. Since $L_a(1) = a$, we conclude that $\mathrm{Der}_0(J)=\mathrm{span}\{[L_a,L_b]|a,b \in J\}$.
    \end{proof}
    
    Let $\mathfrak{m}_J$ denote the subspace of $\mathfrak{g}(J)$ spanned by $L_a$.
    
    \begin{cor}\label{decompofgj}
    	$\mathfrak{g}(J)=\mathfrak{m}_J \oplus [\mathfrak{m}_J,\mathfrak{m}_J]$ for a unital Jordan superalgebra $J$.
    \end{cor}
    \begin{proof}
    	This follows directly from Lemmas \ref{strjor} and \ref{unider}.
    \end{proof}

    \subsection{Peirce decompositions}
    
    Let $J$ be a Jordan superalgebra. Let $e \in J_0$ be an idempotent element, that is, $e^2=e$. Applying \eqref{joride} to $e$ with $b=c=e$ gives
    \begin{align*}
    	[L_a, L_{e}](e)=2[L_{\{a,e\}},L_e](e),
    \end{align*}
    which is equivalent to
    $
    \{e,a\}-3\{e, \{e, a\}\} +2 \{e, \{e, \{e, a\}\} = 0.
    $
    Since $a$ can be any element in $J$, we conclude that
    \begin{align*}
    	2L_e^3 - 3 L_e^2 + L_e = 0.
    \end{align*}
    It follows that $L_e$ is diagonalizable and has eigenvalues $0, \frac{1}{2}, 1$. Let $P_{\lambda}$ denote the eigenspace of $L_e$, $\lambda=0, \frac{1}{2}, 1$. 
    
    \begin{prop}\cite{Martin17}*{Theorem 9}\label{idemdecomp}
    	Let $e \in J_0$ be an idempotent element. The decomposition $J=P_0 \oplus P_{\frac{1}{2}} \oplus P_1$ associated to $e$ has the following properties
    	\begin{enumerate}
    		\item $P_0$ and $P_1$ are subalgebras of $J$ with $\{P_0, P_1\}=\{0\}$;
    		\item $\{P_0, P_{\frac{1}{2}}\} \subseteq P_{\frac{1}{2}}$, $\{P_1, P_{\frac{1}{2}}\} \subseteq P_{\frac{1}{2}}$;
    		\item $\{P_{\frac{1}{2}}, P_{\frac{1}{2}}\} \subseteq P_0 \oplus P_1$.
    	\end{enumerate}
    \end{prop}
    \begin{rmk}\label{pdor}
    	For a pseudo-Euclidean Jordan superalgebra $(J,\beta)$, $P_{\lambda}$ is orthogonal to $P_{\lambda'}$ for $\lambda \neq \lambda'$ because
    	\begin{align*}
    		\lambda \beta(x,y) = \beta(\{e,x\},y) = \beta(x, \{e,y\}) = \lambda' \beta(x,y)
    	\end{align*}
    	for $x \in P_{\lambda}$ and $y \in P_{\lambda'}$.
    \end{rmk}
    
    Let $J$ be a unital Jordan superalgebra. An idempotent element $e \in J_0$ is called primitive if there is no decomposition $e=e_1+e_2$ with idempotent elements $e_1$ and $e_2$. A Jordan frame of $J$ is a collection $\{e_i\}_{i=1}^r$ of primitive idempotent elements such that
    \begin{align*}
    	\{e_i, e_j\} = \delta_{ij} e_i, \quad \sum_{i=1}^r e_i = 1.
    \end{align*}
    Let $a=b=e_i$ and $c=e_j$ in \eqref{joride}, we have $[L_{e_i}, L_{e_j}]=0$. On the other hand, if $x$ is a common eigenvector of all $L_i$ with $L_i x = \lambda_i x$, we must have $\sum_{i=1}^r \lambda_i = 1$. Therefore, there exists a decomposition $J= \bigoplus_{i \leq j} P_{ij}$, where $P_{ij}=\{x \in J: e_i x = e_j x = \frac{1}{2} x\}$ for $i \neq j$ and $P_{ii}=\{x \in J: e_i x = x\}$. (Note that $P_{ij}=P_{ji}$.) This is called the Peirce decomposition of $J$ with respect to the Jordan frame $\{e_i\}_{i=1}^r$. Using Proposition \ref{idemdecomp}, it is easy to show that
    \begin{prop}\label{peirce}
    	Let $J$ be a unital Jordan superalgebra with Jordan frame $\{e_i\}_{i=1}^r$. The decomposition $J= \bigoplus_{i \leq j} P_{ij}$ associated to the Jordan frame satisfies that following properties.
    	\begin{enumerate}
    		\item $P_{ii}$ are subalgebras of $J$ for all $i$ with $\{P_{ii}, P_{jj}\}=0$ for $i \neq j$;
    		\item $\{P_{ii},P_{ij}\} \subseteq P_{ij}$ for $i \neq j$, $\{P_{ii},P_{jk}\} = \{0\}$ when $i, j, k$ are all distinct;
    		\item $\{P_{ij},P_{ij}\} \subseteq P_{ii} \oplus P_{jj}$ for $i\neq j$, $\{P_{ij},P_{jk}\} \subseteq P_{ik}$ when $i, j, k$ are all distinct, and $\{P_{ij},P_{kl}\} = \{0\}$ when $i, j, k, l$ are all distinct. 
    	\end{enumerate}
    \end{prop}
    
    \subsection{Euclidean Jordan superalgebras}
    
    \begin{prop}\label{euclijor}
    	Let $J$ be a Jordan algebra. The following statements hold.
    	\begin{enumerate}[ref={Point~\arabic*}]
    		\item \label{1} $J$ is semi-simple if and only if it admits a (unique) decomposition $J=\bigoplus_i I_i$ where $I_i$ are simple ideals of $J$.
    		\item \label{2} There exists a bilinear form $\beta$ on $J$ such that $(J,\beta)$ is Euclidean if and only if $J$ is postive.
    		\item \label{3} If $J$ is positive, there exists a Jordan frame $\{e_i\}_{i=1}^r$ such that $x = \sum_{i=1}^r \lambda_i e_i$ for every $x \in J$. 
     	\end{enumerate}
    \end{prop}
    
    \begin{rmk}\label{idssss}
    	A direct consequence of \ref{1} of Proposition \ref{euclijor} is that every (non-trivial) ideal of a semi-simple Jordan algebra is semi-simple.
    \end{rmk}

    \begin{rmk}
    	\ref{1} and \ref{2} of Proposition \ref{euclijor} are no longer true in the super case. Consider the Jordan superalgebra $K_3$ defined by $(K_3)_0 = \mathbb{R} e$, $(K_3)_1 = \mathbb{R} x \oplus \mathbb{R} y$ with $\{e,e\}=e$, $\{e,x\}=\frac{1}{2}x$, $\{e,y\}=\frac{1}{2}y$, and $\{x,y\}=e$. $K_3$ is called the Kaplansky superalgebra. $K_3$ is simple. However, its canonical bilinear form $\tau$ is zero. Moreover, there exists a positive definite associative bilinear form $\beta$ on $K_3$ defined by setting
    	\begin{align*}
    		\beta(e,e)=1, \quad \beta(e,x)=\beta(e,y)=0,\quad \beta(x,y)=2.
    	\end{align*}
    	In other words, $K_3$ is Euclidean but not positive.
    \end{rmk}
    
    \begin{figure}[!h]
    	\centering
    	\begin{subfigure}{0.4\textwidth}
    		\begin{tikzpicture}[set/.style={fill=orange, opacity = 0.1, text opacity = 1}]  		
    			\node at (0,2.5) {pseudo-Euclidean};
    			\node at (0,1.5) {Semi-simple};
    			\node at (0,0) {Positive/Euclidean};
    			
    			\draw[set,
    			rotate =0] (0,0) ellipse (3.3cm and 3.3cm);
    			\draw[set,
    			rotate =0] (0,0) ellipse (2.5cm and 2cm);
    			\draw[set,
    			rotate =0] (0,0) ellipse (2cm and 1cm);    		
    		\end{tikzpicture}
    		\caption{Non-super case.}
    	\end{subfigure}
    	\hspace{1cm}
    	 \begin{subfigure}{0.4\textwidth}
    	 	\begin{tikzpicture}[set/.style={fill=cyan, opacity = 0.1, text opacity = 1}]  		
    	 		\node at (0,2.7) {pseudo-Euclidean};
    	 		\node at (0,-1.8) {Semi-simple};
    	 		\node at (0,1.8) {Euclidean};
    	 		\node at (0,0) {Positive};
    	 		
    	 		\node at (1.5,1.2) {$K_3$};
    	 		\filldraw (1.2,1.5) circle [radius=0.05cm];
    	 		
    	 		\draw[set,
    	 		rotate =0] (0,0) ellipse (3.3cm and 3.3cm);
    	 		\draw[set,
    	 		rotate =0] (0,-0.5) ellipse (2cm and 1.8cm);
    	 		\draw[set,
    	 		rotate =0] (0,0.5) ellipse (2cm and 1.8cm);
    	 		\draw[set,
    	 		rotate =0] (0,0) circle (1cm);    		
    	 	\end{tikzpicture}
    	 	\caption{Super case.}
    	 \end{subfigure}
    	 \caption{Venn diagram showing the relationship between pseudo-Euclidean, Euclidean, semi-simple, and positive Jordan (super)algebras.}
    \end{figure}
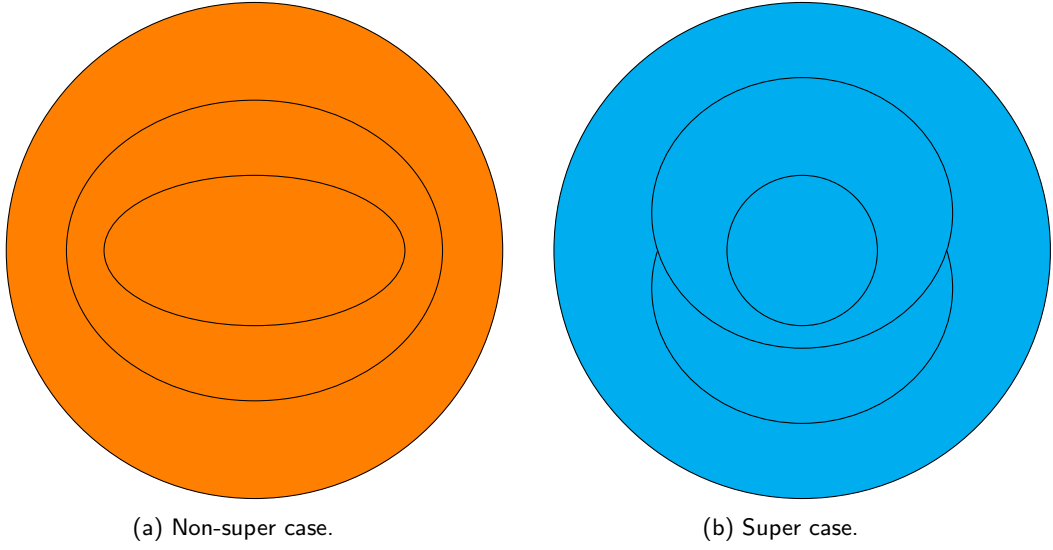
    
     \begin{prop}\label{eucuni}
    	Every positive Jordan superalgebra $J$ is unital.
    \end{prop}
    \begin{proof}
    	If $J$ is a positive Jordan superalgebra, by \ref{2} of Proposition \ref{euclijor}, the Jordan algebra $J_0$ is also postive, hence possesses a unit element $1$ \cite{Koecher99}. Let $a, b \in J_1$, we have
    	\begin{align*}
    		\tau(1a-a,b)=\tau(1,ab)-\tau(a,b)=\mathrm{str}(L_{1(ab)})-\mathrm{str}(L_{ab})=0.
    	\end{align*}
    	Since $\tau$ is non-degenerate, $1 \in J_0$ is also the unit element of $J$.
    \end{proof}
    \begin{rmk}
    	Euclidean Jordan superalgebras are not necessarily unital. (E.g., $K3$.) 
    \end{rmk}

    
    \begin{prop}\label{supereuclidecomp}
    	Every unital Euclidean Jordan superalgebra $(J,\beta)$ admits a (unique) decomposition $J=\bigoplus_i I_i$, where $I_i$ are simple Euclidean ideals of $J$.
    \end{prop}
    \begin{proof}
    	Consider the decomposition $J=\bigoplus_i I_i$ in Proposition \ref{betadecomp}. Each $I_i$ equipped with $\beta|_{I_i}$ is Euclidean. By \ref{2} of Proposition \ref{euclijor}, the even part $(I_i)_0$ of $I_i$ positive.
    	
    	Let $K$ be a nonzero ideal of $I_i$. The even part $K_0$ of $K$ is an ideal of $(I_i)_0$. By Remark \ref{idssss}, $K_0$ is semi-simple and has a unit element $e$. In general, $e$ is an idempotent element in $I_i$.  Consider the Peirce decomposition $I_i=P_0 \oplus P_{\frac{1}{2}} \oplus P_1$ associated to $e$. Let $x \in P_{\lambda}$, $\lambda = \frac{1}{2}, 1$. We have $x = \frac{1}{\lambda} \{e,x\}$. Thus, $P_{\frac{1}{2}} \oplus P_1 \subset K$. On the other hand, $K_0 \subset P_1$ since $\{e,x\}=x$ for all $x \in K_0$. We conclude that 
    	$
    		(P_{\frac{1}{2}})_0=\{0\},~ K_0 = (P_1)_0, 
    	$
    	which then imply that 
    	$
    		(I_i)_0=K_0 \oplus (P_0)_0.
    	$
    	By Remark \ref{pdor}, $\beta|_{K_0}$ must be non-degenerate. 
    	
    	Suppose that $K_1 \cap (P_0)_1$ is non-empty. Since $\beta|_{I_i}$ is non-degenerate, for a nonzero element $y \in K_1 \cap (P_0)_1$, one can find an element $z \in (P_0)_1$ such that $\beta(y,z) \neq 0$. By Proposition \ref{idemdecomp}, $\{y,z\} \in (P_0)_0$. Since $K_0=(P_1)_0$, $\{y,z\} = 0$. But this cannot be because $0 \neq \beta(y,z) = \beta(1|_{I_i},\{y,z\})$, where $1|_{I_i}$ is the unit element of $I_i$. Therefore, $K_1 = (P_{\frac{1}{2}} \oplus P_1)_1$ and $\beta|_{K_1}$ is also non-degenerate. By Proposition \ref{betadecomp}, $K$ must be $I_i$ itself.
    	
    	The proof of the uniqueness of this decomposition is standard.
    \end{proof}

    \begin{cor}
    	The structure Lie superalgebra of a unital Euclidean Jordan superalgebra $(J,\beta)$ is the direct product of the structure Lie superalgebras of the simple ideals of $J$.
    \end{cor}

    \begin{prop}
    	Let $(J,\beta)$ be a unital pseudo-Euclidean Jordan superalgebra with positive even part. For every $x \in J_0$,  there exists a Jordan frame $\{e_i\}_{i=1}^r$ such that $x = \sum_{i=1}^r \lambda_i e_i$. 
    	The tuple $\{\lambda_i\}_{i=1}^r$ is defined only up to permutations of the entries.
    \end{prop}
    \begin{proof}
    	This follows directly from \ref{3} of Proposition \ref{euclijor}.
    \end{proof}
    
    \section{Jordan distributions}
    
    In this section, we study the generalized distribution over the dual $J^*$ of a unital Jordan superalgebra $J$ induced by the action of the structure Lie supergroup $G(J)$ on $J^*$.
    
    \subsection{Main theorems}
    
    For convenience, let's first consider the $G(J)$-action on $J$. To compute the tangent space of a $G(J)$-orbit $\mathcal{O}$ at $x \in J_0$, we need the following lemma, which is a generalization of \cite{Ciaglia23}*{Lemma 3}.
    \begin{lem}
    	Let $x = \sum_{i=1}^r \lambda_i e_i \in J_0$ with $\{e_i\}_{i=1}^r$ be a Jordan frame. Then
    	\begin{align*}
    		\mathfrak{m}_J \cdot x = \bigoplus_{\lambda_i + \lambda_j \neq 0} P_{ij}, \quad
    		\mathrm{Der}_0(J) \cdot x = \bigoplus_{\lambda_i - \lambda_j \neq 0} P_{ij}, \quad
    		\mathfrak{g}(J) \cdot x = \bigoplus_{\lambda_i \neq 0~\mathrm{or}~\lambda_j \neq 0} P_{ij},
    	\end{align*}
        where the direct sums are implicitly assumed to be taken over $1 \leq i \leq j \leq r$.
    \end{lem}
    \begin{proof}
    	The first equality follows from $L_{x_{ij}}(x)=\frac{1}{2}(\lambda_i + \lambda_j) x_{ij}$ for $x_{ij} \in P_{ij}$. For the second equality, we have
    	\begin{align}\label{innperx}
    		[L_{x_{ij}},L_{y_{kl}}](x) &= \{x_{ij}, \frac{1}{2}(\lambda_k + \lambda_l)\} -(-1)^{|x_{ij}||y_{kl}|}\{y_{kl}, \frac{1}{2}(\lambda_i + \lambda_j) x_{ij}\} \\
    		&= \frac{1}{2}(\lambda_k + \lambda_l - \lambda_i - \lambda_j)\{x_{ij},y_{kl}\}
    	\end{align}
        for $x_{ij} \in P_{ij}$ and $y_{kl} \in P_{kl}$. By Proposition \ref{peirce}, $\{x_{ij},y_{kl}\}$ can be nonzero only in the following cases.
        \begin{enumerate}
        	\item $i = j = k = l$, $[L_{x_{ij}},L_{y_{kl}}](x)=\frac{1}{2}(\lambda_i + \lambda_i - \lambda_i - \lambda_i)\{x_{ii},y_{ii}\}=0$;
        	\item $i=j=k \neq l$, $[L_{x_{ij}},L_{y_{kl}}](x)=\frac{1}{2}(\lambda_i + \lambda_l - \lambda_i - \lambda_i) \{x_{ii},y_{il}\} = \frac{1}{2}(\lambda_l - \lambda_i) \{x_{ii},y_{il}\} \in P_{il}$;
        	\item $i=k \neq j=l$, $[L_{x_{ij}},L_{y_{kl}}](x)=\frac{1}{2}(\lambda_i + \lambda_j - \lambda_i - \lambda_j)\{x_{ij},y_{ij}\}=0$;
        	\item $i \neq j = k \neq l$, $i \neq l$, $[L_{x_{ij}},L_{y_{kl}}](x)=\frac{1}{2}(\lambda_j + \lambda_l - \lambda_i - \lambda_i) \{x_{ij},y_{jl}\} = \frac{1}{2}(\lambda_l - \lambda_i) \{x_{ij},y_{jl}\} \in P_{il}$.
        \end{enumerate} 
        It follows that $\mathrm{Der}_0(J) \cdot x \subset \bigoplus_{\lambda_i - \lambda_j \neq 0} P_{ij}$. For the converse inclusion, we compute
        \begin{align*}
        	x_{ij}=-\frac{4}{(\lambda_i - \lambda_j)^2}[L_x, L_{x_{ij}}](x)
        \end{align*}
        when $\lambda_i - \lambda_j \neq 0$.
    \end{proof}
    
    Let $\mathcal{O}_0$ denote the underlying topological space of $\mathcal{O}$.
    \begin{prop}
    	For $x \in \mathcal{O}_0 \subset J_0$, $T_x \mathcal{O} \cong \bigoplus_{\lambda_i \neq 0~\mathrm{or}~\lambda_j \neq 0} P_{ij}$.
    \end{prop}
    
    What is $\mathcal{O}_0$ in the super case? 
    By Corollary \ref{decompofgj}, we have
    \begin{align}\label{gj0}
    	\mathfrak{g}(J)_0 = (\mathfrak{m}_J \oplus [\mathfrak{m}_J,\mathfrak{m}_J])_0 = \mathfrak{g}(J_0) \oplus [\mathfrak{m}_{J_1},\mathfrak{m}_{J_1}].
    \end{align}
    Recall that elements in $ [\mathfrak{m}_{J},\mathfrak{m}_{J}]$ are the inner derivations of $J$.
    \begin{thm}
    	The underlying topological space $\mathcal{O}_0$ of a $G(J)$-orbit $\mathcal{O}$ coincides with the corresponding $G(J_0)$-orbit in $J_0$. 
    \end{thm}
    \begin{proof}
    	It follows from \eqref{gj0} that $G(J)_0$ is generated by elements in $\mathfrak{g}(J_0)$ and certain derivations of $J_0$. On the other hand, by Theorem \ref{bosorb}, $G(J_0)$-orbits are preserved by automorphisms of $J_0$.
    \end{proof}
    
    Let's recall the following theorem for positive Jordan algebras.
    \begin{thm}\cite{Ciaglia23}*{Theorem 1}\label{bosorb}
    	For a positive Jordan algebra $J$, the following holds:
    	\begin{enumerate}
    		\item The orbits of $\mathrm{Aut}(J)_0$ are the sets of elements with equal spectral coefficients, where $\mathrm{Aut}(J)_0 \subset \mathrm{gl}(J)$ is the connected Lie group generated by derivations of $J$.
    		\item The orbits of the structure Lie group $G(J)$ are the sets of elements with equal spectral signature.
    	\end{enumerate}
    \end{thm}
    
    \begin{cor}
    	For a unital Jordan superalgebra $J$ with positive even part $J_0$, the underlying topological space $\mathcal{O}_0$ of the $G(J)$-orbit $\mathcal{O}$ consists of elements in $J_0$ with equal spectral signatures.
    \end{cor}
    
    Let's now study the action of $G(J)$ on  $J^*$. Let $(J,\beta)$ be a unital pseudo-Euclidean Jordan superalgebra with positive even part. $\beta$ induces the following identification
    \begin{align*}
    	\flat: J \rightarrow J^*, \quad x \mapsto x^{\flat}:=\beta(x,\cdot).
    \end{align*} 
    Let $\sharp: J^* \rightarrow J$ denote the inverse of $\flat$. Let $L_x^*$ denote the dual action of $L_x$ on $J^*$. We have
    \begin{align*}
    	L_x^*(\xi)(y)=\xi(L_x(y)) = \beta(\xi^{\sharp},L_x(y)) = (-1)^{|x||\xi|}\beta(L_x(\xi^{\sharp}),y)=(-1)^{|x||\xi|}(L_x(\xi^{\sharp}))^{\flat}(y).
    \end{align*}
    Hence,
    \begin{align*}
    	L_x^*(\xi) = (-1)^{|x||\xi|}(L_x(\xi^{\sharp}))^{\flat}.
    \end{align*}
    With a slight abuse of notation, we use again $\mathcal{O}$ denote a $G(J)$-orbit in $J^*$.  For $\xi \in J_0^*$, we can define the spectral coefficients of $\xi$ to be the spectral coefficients of $\xi^{\sharp}$. 
     \begin{prop}
    	For $\xi \in \mathcal{O}_0 \subset J_0^*$, $T_{\xi} \mathcal{O} \cong \bigoplus_{\lambda_i \neq 0~\mathrm{or}~\lambda_j \neq 0} P_{ij}^{\flat}$.
    \end{prop}
    
    \begin{thm}
         Let $(J,\beta)$ be a unital pseudo-Euclidean Jordan superalgebra with positive even part. The underlying topological space $\mathcal{O}_0$ of a $G(J)$-orbit $\mathcal{O}$ consists of elements with equal spectral signatures.
    \end{thm}
    
    \begin{thm}
    	Let $(J,\beta)$ be a unital pseudo-Euclidean Jordan superalgebra with positive even part. $\xi \in J^*_0$ is $\mathfrak{m}_J$-regular if and only if the spectral coefficients $\{\lambda_i\}_{i=1}^r$ of $\xi$ satisfy:
    	\begin{align*}
    		\lambda_i + \lambda_j \neq 0, \quad \forall \lambda_i, \lambda_j \neq 0.
    	\end{align*}
    \end{thm}
    \begin{proof}
    	The proof is essentially the same as the proof of \cite{Ciaglia23}*{Theorem 2}.
    \end{proof}
   
   Let $\xi \in J_0^*$. Supposing that $\xi^{\sharp}=\sum_{i=1}^r \lambda_i e_i$, where $\{e_i\}_{i=1}^r$ is a Jordan frame of $J_0$, then $\xi=(\xi^{\sharp})^{\flat} = \sum_{i=1}^r \lambda_i e_i^{\flat}$. 
   \begin{prop}\label{gb}
   	Let $\xi = \sum_{i=1}^r \lambda_i e_i^{\flat} \in J_0^*$ be as above. Then the bilinear pairing \eqref{g} at $\xi$ can be expressed as
   	\begin{align*}
   		\mathcal{g}_{\xi}= \sum_{\lambda_i \neq 0~\mathrm{or}~\lambda_j \neq 0} \frac{2}{\lambda_i + \lambda_j} \beta^{\flat}|_{P_{ij}^{\flat}},
   	\end{align*}
   	where $\beta^{\flat}$ is defined as $\beta^{\flat}(\xi_1,\xi_2):=\beta(\xi_1^{\sharp}, \xi_2^{\sharp})$, and the summation is implicitly assumed to be taken over $1 \leq i \leq j \leq r$.
   \end{prop}
   \begin{proof}
   	The proof is essentially the same as the proof of \cite{Ciaglia23}*{Proposition 5}.
   \end{proof}
    
    \subsection{Examples}
    
    Let $A$ be an associative superalgebra over a field $\mathbb{K}$ of characteristic 0. (We will mostly choose $\mathbb{K}$ to be $\mathbb{R}$.) The algebraic product on the underlying super vector space $A$ given by
    \[
    \{a,b\}=\frac{1}{2}\left(a b+(-1)^{|a||b|} b a\right)
    \]
    defines a Jordan superalgebra structure on $A$ which is denoted $A^{(+)}$. Jordan superalgebras that can be obtained as subalgebras of $A^{(+)}$ are called special. Jordan superalgebras that are not special are called exceptional.
    
    Recall that $\mathrm{gl}(\mathbb{K}^{m|n})$ is an associative superalgebra with even and odd parts
    \begin{align*}
    	\mathrm{gl}(\mathbb{K}^{m|n})_{0}=\bigg\{\begin{pmatrix}
    		a & 0\\
    		0 & d
    	\end{pmatrix}\bigg| a \in \mathrm{gl}(\mathbb{K}^m), d \in \mathrm{gl}(\mathbb{K}^n)\bigg\}, 
    	\quad 
    	\mathrm{gl}(\mathbb{K}^{m|n})_{1}=\bigg\{\begin{pmatrix}
    		0 & b\\
    		c & 0
    	\end{pmatrix}\bigg\}.
    \end{align*}
    $\mathrm{gl}(\mathbb{K}^{m|n})^{(+)}$ is a Jordan superalgebra. There exists an associative bilinear form $\beta$ on $\mathrm{gl}(\mathbb{K}^{m|n})^{(+)}$ defined by setting
    \begin{align*}
    	\beta(A,B):=\mathrm{str}(AB).
    \end{align*}
    \begin{exmp}\label{a11}
    	Let's consider the case where $n=m=1$, $\mathbb{K}=\mathbb{R}$. We choose the following basis
    	\begin{align*}
    		e_1 := \begin{pmatrix}
    			1  & 0 \\
    			0 & 0 
    		\end{pmatrix}, \quad
    		e_2 := \begin{pmatrix}
    			0  & 0 \\
    			0 & 1 
    		\end{pmatrix}, \quad
    		x := \begin{pmatrix}
    			0  & \sqrt{2}\\
    			0 & 0
    		\end{pmatrix}, \quad
    		y := \begin{pmatrix}
    			0  & 0 \\
    			\sqrt{2} & 0 
    		\end{pmatrix},    		
    	\end{align*}
    	where $e_1$, $e_2$ are even, and $x$, $y$ are odd. One can easily verify that $\{e_i\}_{i=1}^2$ is a Jordan frame and
    	\begin{align*}
    		\{e_1, x\}=\{e_2, x\}=\frac{1}{2} x, \quad \{e_1, y\}=\{e_2,y\}=\frac{1}{2}y.
    	\end{align*}
    	Therefore, the Peirce decomposition is given by $P_{11}=\mathrm{span}_{\mathbb{R}}\{e_1\}, P_{22}=\mathrm{span}_{\mathbb{R}}\{e_2\}$, $P_{12}=\mathrm{span}_{\mathbb{R}}\{x,y\}$. One can also check that
    	\begin{align*}
    		\{x,y\}=\frac{1}{2}\left(\begin{pmatrix}
    			2 & 0\\
    			0 & 0 
    		\end{pmatrix} - \begin{pmatrix}
    			0 & 0\\
    			0 & 2 
    		\end{pmatrix} \right) = e_1  - e_2.
    	\end{align*}
    	It follows that $\tau \equiv 0$. $\mathrm{gl}(\mathbb{R}^{1|1})^{(+)}$ is not semi-simple. However, one can check that
    	\begin{align*}
    		\beta(e_1,e_1)= 1, \quad \beta(e_2,e_2)=-1, \quad \beta(x,y)=2.
    	\end{align*}
    	Therefore, $(\mathrm{gl}(\mathbb{R}^{1|1})^{(+)},\beta)$ is pseudo-Euclidean with signature $(1,1)$.
    \end{exmp}
    
    A map $*: A \rightarrow A$ is a (super)involution if it satisfies $\left(a^{*}\right)^{*}=a$ and $(a b)^{*}=(-1)^{|a||b|} b^{*} a^{*}$.
    The set of symmetric elements $H(A, *):=\{a^*=a| a \in A\}$ is a special Jordan superalgebra since $\{a,b\}^*=\frac{1}{2}\left((-1)^{|a||b|}b^*a^* + a^*b^*\right) = \{a,b\}$ for $a, b \in H(A, *)$.
    Let $A=\mathrm{gl}(\mathbb{K}^{m|2n})$. Let $\mathrm{Id}_m$ and $\mathrm{Id}_n$ denote the $m \times m$ and $n \times n$ identity matrices, respectively. Let $U$ denote $\begin{pmatrix}
    	0 & -\mathrm{Id}_n \\
    	\mathrm{Id}_n & 0
    \end{pmatrix}$. One can check that the map $*: A \rightarrow A$ defined by
    \begin{align*}
    	\begin{pmatrix}
    		a & b\\
    		c & d
    	\end{pmatrix}
    	\mapsto
    	\begin{pmatrix}
    		\mathrm{Id}_m & 0\\
    		0 & U
    	\end{pmatrix}
    	\begin{pmatrix}
    		a^t & -c^t \\
    		b^t & d^t
    	\end{pmatrix}
    	\begin{pmatrix}
    		\mathrm{Id}_m & 0\\
    		0 & U^{-1}
    	\end{pmatrix}
    \end{align*}
    is a superinvolution, where $(\cdot)^t$ is the usual transpose operation of matrices. The Jordan superalgebra $\mathrm{Josp}(\mathbb{K}^{m|2n}):=H(\mathrm{gl}(\mathbb{K}^{m|2n}), *)$ is called the Jordan orthosymplectic superalgebra. Explicitly, elements of $\mathrm{Josp}(\mathbb{K}^{n|2m})$ are matrices of the form
    \begin{align*}
    	\begin{pmatrix}
    		a & b_1 & b_2 \\
    		-b_2^t & d_{11} & d_{12} \\
    		b_1^t & d_{21} & d_{11}^t
    	\end{pmatrix},
    \end{align*}
    where $a$ is a symmetric $m \times m$ matrix, $b_1$, $b_2$ are $m \times n$ matrices, $d_{11}$ is an arbitrary $n \times n$ matrix, and $d_{12}$, $d_{21}$ are $n \times n$ skew-symmetric matrices.
    
    \begin{exmp}\label{josp}
    	Let's consider the case where $n=m=1$, $\mathbb{K}=\mathbb{R}$. We choose the following basis
    	\begin{align*}
    		e_1 := \begin{pmatrix}
    			1  & 0 & 0 \\
    			0 & 0 & 0 \\
    		    0 & 0 & 0
    			\end{pmatrix}, \quad
    	    e_2 := \begin{pmatrix}
    	    	0  & 0 & 0 \\
    	    	0 & 1 & 0 \\
    	    	0 & 0 & 1
    	    \end{pmatrix}, \quad
    	    x := \begin{pmatrix}
    	    	0  & 1 & 0 \\
    	    	0 & 0 & 0 \\
    	    	1 & 0 & 0
    	    \end{pmatrix}, \quad
    	    y := \begin{pmatrix}
    	    	0  & 0 & -1 \\
    	    	1 & 0 & 0 \\
    	    	0 & 0 & 0
    	    \end{pmatrix}    		
    	\end{align*}
    	of $\mathrm{Josp}(\mathbb{R}^{1|2})$, where $e_1$, $e_2$ are even, and $x$, $y$ are odd. One can easily verify that $\{e_i\}_{i=1}^2$ is a Jordan frame and
    	\begin{align*}
    		\{e_1, x\}=\{e_2, x\}=\frac{1}{2} x, \quad \{e_1, y\}=\{e_2,y\}=\frac{1}{2}y.
    	\end{align*}
    	Therefore, the Peirce decomposition is given by $P_{11}=\mathrm{span}_{\mathbb{R}}\{e_1\}, P_{22}=\mathrm{span}_{\mathbb{R}}\{e_2\}$, $P_{12}=\mathrm{span}_{\mathbb{R}}\{x,y\}$. One can also check that
    	\begin{align*}
    		\{x,y\}=\frac{1}{2}\left(\begin{pmatrix}
    			1 & 0 & 0 \\
    			0 & 0 & 0 \\
    			0 & 0 & -1
    		\end{pmatrix} - \begin{pmatrix}
    		-1 & 0 & 0\\
    		0 & 1 & 0 \\
    		0 & 0  & 0
    		\end{pmatrix} \right) = e_1  - \frac{1}{2} e_2.
    	\end{align*}
    	It follows that $\tau \equiv 0$. $\mathrm{Josp}(\mathbb{R}^{1|2})$ is not semi-simple. However, one can check that
    	\begin{align*}
    		\beta(e_1,e_1) = 1, \quad \beta(e_2,e_2)=-2, \quad \beta(x,y)=2.
    	\end{align*}
    	Therefore, $(\mathrm{Josp}(\mathbb{R}^{1|2}),\beta)$ is a pseudo-Euclidean Jordan superalgebra with signature $(1,1)$.
    \end{exmp}
    \begin{rmk}
    	If $\mathbb{K}=\mathbb{C}$, we can replace the transpose operation $(\cdot)^{t}$ by the conjugate transpose operation $(\cdot)^{\dagger}$ in the definition of $\mathrm{Josp}(\mathbb{K}^{m|2n})$. We denote the resulting real Jordan superalgebras by $\mathrm{UJosp}(m,2n)$. $\mathrm{UJosp}(m,0)$ is simply the Jordan algebra $\mathcal{B}_{sa}(\mathbb{C}^m)$ of $m \times m$ hermitian matrices.
    \end{rmk}
    Let $V$ be a super vector space of dimension $(p|q)$, equipped with a even, symmetric, non-degenerate bilinear form $\langle \cdot, \cdot \rangle$. Then $\mathrm{Spin}(\mathbb{K}^{p|q}):=\mathbb{K}1+V$ is a Jordan superalgebra, where the product of two elements $a 1+u$ and $b 1+v$ in $J$ is given by
    \[
    \{a 1+u, b 1+v \}=a b 1+\langle u, v \rangle 1+ a v + b u
    \]
    for $u, v \in V$.
    $\mathrm{Spin}(\mathbb{K}^{p|q})$ is simple if and only if the form $\langle \cdot,\cdot \rangle$ is non-degenerate. Let $\beta$ be a bilinear form on $\mathrm{Spin}(\mathbb{K}^{p|q})$ given by
    \begin{align*}
    	\beta(a 1+u, b 1+v):= 2\left(a b + \langle u, v \rangle\right).
    \end{align*}
    It is easy to verify that $\beta$ is associative. In fact, we have
    \begin{align*}
    	\beta(\{a 1 +u, b 1 + v\}, c 1 + w) = \beta(a 1 + u, \{b 1 + v, c 1 + w\}) = 4\left( abc + a \langle v, w \rangle + b \langle u,w \rangle + c \langle u,v \rangle \right).
    \end{align*}

    \begin{exmp}\label{spin30}
    	Let's consider the case where $p=3$, $q=0$, and $\mathbb{K}=\mathbb{R}$. The algebra $\mathrm{Spin}(\mathbb{R}^{3|0})$ is purely even. Let $\{e,x,y\}$ be an orthonormal basis of $V = V_0$. Let
    	\begin{align*}
    		e_1 := \frac{1+e}{2}, \quad e_2:= \frac{1-e}{2}.
    	\end{align*}
    	One can easily verify that $\{e_i\}_{i=1}^2$ is a Jordan frame and
    	\begin{align*}
    		\{e_1, x\}=\{e_2, x\}=\frac{1}{2} x, \quad \{e_1, y\}=\{e_2,y\}=\frac{1}{2}y.
    	\end{align*}
    	Therefore, the Peirce decomposition is given by $P_{11}=\mathrm{span}_{\mathbb{R}}\{e_1\}, P_{22}=\mathrm{span}_{\mathbb{R}}\{e_2\}$, $P_{12}=\mathrm{span}_{\mathbb{R}}\{x,y\}$. One can also check that
    	\begin{align*}
    		\{x,x\}= \{y,y\} = e_1 + e_2, \quad \{x,y\}=0.
    	\end{align*}
        It is not hard to check that $\mathrm{Spin}(\mathbb{R}^{3|0}) \cong \mathrm{UJosp}(2,0)$ as real Jordan algebras.   
    \end{exmp}

    \begin{exmp}\label{spin12}
    	Let's consider the case where $p=1$, $q=2$. Let $e \in V_0$ be a unit vector, i.e., $\langle e,e\rangle=1$. Let $x, y \in V_1$ be two non-zero vectors such that $\langle x, y \rangle = 1$. Let
    	\begin{align*}
    		e_1 := \frac{1+e}{2}, \quad e_2:= \frac{1-e}{2}.
    	\end{align*}
    	 One can easily verify that $\{e_i\}_{i=1}^2$ is a Jordan frame and
    	\begin{align*}
    		\{e_1, x\}=\{e_2, x\}=\frac{1}{2} x, \quad \{e_1, y\}=\{e_2,y\}=\frac{1}{2}y.
    	\end{align*}
    	Therefore, the Peirce decomposition is given by $P_{11}=\mathrm{span}_{\mathbb{R}}\{e_1\}, P_{22}=\mathrm{span}_{\mathbb{R}}\{e_2\}$, $P_{12}=\mathrm{span}_{\mathbb{R}}\{x,y\}$. One can also check that
    	\begin{align*}
    		\{x,y\}= \langle x, y \rangle 1 = e_1 + e_2.
    	\end{align*}
    	One has $\tau \equiv 0$. $\mathrm{Spin}(\mathbb{R}^{1|2})$ is not semi-simple. However, one can check that
    	\begin{align*}
    		\beta(e_1,e_1) = \beta(e_2,e_2)=1,  \quad \beta(x,y)=2.
    	\end{align*}
    	Therefore, $(\mathrm{Spin}(\mathbb{R}^{1|2}), \beta)$ is Euclidean.
    \end{exmp}
    
    There exists an one-parametric family of $(2|2)$-dimensional Jordan superalgebras which generalizes Examples \ref{a11}, \ref{josp}, and \ref{spin12}.
    
    \begin{exmp}\label{dt}
    	Let $D(t):=\mathrm{span}_{\mathbb{K}}\{e_1,e_2\} \oplus \mathrm{span}_{\mathbb{K}}\{x,y\}$ with the product
    	\begin{align*}
    		\{e_i,e_i\}=e_i, \quad \{e_{1}, e_{2}\}=0, \quad \{e_{i}, x\}=\frac{1}{2} x, \quad \{e_{i}, y\}=\frac{1}{2} y, \quad \{x, y\}=e_{1}+t e_{2}, \quad i=1,2,
    	\end{align*}
    	where $t \in \mathbb{K}$. Note that $D(t)$ is simple if $t \neq 0$, $D(t) \cong D(\frac{1}{t})$, and that
    	\begin{enumerate}
    		\item For $t=-1$, $D(-1)\cong \mathrm{gl}(\mathbb{K}^{1|1})^{(+)}$. 
    		\item For $t=-\frac{1}{2}$, $D(-\frac{1}{2})\cong \mathrm{Josp}(\mathbb{K}^{1|2})$. 
    		\item For $t=0$, $D(0)$ has an ideal that is isomorphic to $K_3$. 
    		\item For $t=1$, $D(1) \cong \mathrm{Spin}(\mathbb{K}^{1|2})$.
    	\end{enumerate}
    	
    	Once again, one can check that $\tau \equiv 0$ for all $t \in \mathbb{K}$. $D(t)$ is not semi-simple. However, one can consider the even, symmetric, non-degenerate bilinear form $\beta$ on $D(t)$, $t \neq 0$, defined by
    	\begin{align*}
    		\beta(e_1,e_1) = 1, \quad \beta(e_2, e_2) = \frac{1}{t}, \quad \beta(x,y)=2.
    	\end{align*} 
    	It is straightforward to verify that $\beta$ is associative. Let $\mathbb{K}=\mathbb{R}$. $(D(t),\beta)$ is Euclidean for $t >0$ and pseudo-Euclidean with signature $(1,1)$ for $t<0$.
    \end{exmp}

    \begin{rmk}
    	One may wonder if it is possible to define a family of real Jordan algebras by setting $D_{ns}(t):=\mathrm{span}_{\mathbb{R}}\{e_1, e_2, x, y\}$ and 
    	\begin{align*}
    		&\{e_i,e_i\}=e_i, \quad \{e_{1}, e_{2}\}=0, \quad \{e_{i}, x\}=\frac{1}{2} x, \quad \{e_{i}, y\}=\frac{1}{2} y, \\
    		&\{x, x\} = \{y,y\} =e_{1}+t e_{2}, \quad  \{x,y\}=0, \quad i=1,2,
    	\end{align*}
    	such that $D_{ns}(1) \cong \mathrm{Spin}(\mathbb{R}^{3|0}) \cong \mathrm{UJosp}(2,0)$. This is not possible because 
    	\begin{align*}
    		\{\{x,e_1\},\{x,x\}\}=\frac{1+t}{4}, \quad \{x,\{e_1,\{x,x\}\}\} = \frac{1}{2}.
    	\end{align*}
    	The Jordan identity then forces us to set $t=1$. In other words, passing to the super setting gives us more freedom to do a ``deformation".
    \end{rmk}

    Let's take $\mathbb{K}=\mathbb{R}$ and compute the value of the pseudo-Riemannian metric $\mathcal{g}$ of the orbit-distribution on $D(t)^*$ at each point $\xi = \lambda_1 e_1^{\flat} + \lambda_2 e_2^{\flat} \in D(t)_0^*$. By Proposition \ref{gb}, one has
    \begin{align*}
    	\mathcal{g}_{\xi}(e_1^{\flat}, e_1^{\flat}) = \frac{2}{\lambda_1 + \lambda_1} \beta(e_1,e_1) &= \frac{1}{\lambda_1}, 
    	\quad
    	\mathcal{g}_{\xi}(e_2^{\flat}, e_2^{\flat}) = \frac{2}{\lambda_2 + \lambda_2} \beta(e_2,e_2) = \frac{1}{t \lambda_2}, \\
    	\mathcal{g}_{\xi}(x^{\flat}, y^{\flat}) &= \frac{2}{\lambda_1 + \lambda_2} \beta(x,y) = \frac{4}{\lambda_1 + \lambda_2}.
    \end{align*}
    Let $(e_1^*, e_2^*, x^*, y^*)$ denote the dual basis of $D(t)^*$. One has
    \begin{align*}
    	e_1^{\flat} = e_1^*, \quad e_2^{\flat} = \frac{1}{t} e_2^*, \quad x^{\flat} = 2y^*, \quad y^{\flat} = -2x^*.
    \end{align*}
    Let $\xi = \lambda_1 e_1^{\flat} + \lambda_2 e_2^{\flat} = \lambda_1' e_1^* + \lambda_2' e_2^*$, $\eta = z_1 e_1^* + z_2 e_2^* + w_1 x^*  + w_2 y^* $, and $\eta' = z_1' e_1^* + z_2' e_2^* + w_1' x^*  + w_2' y^*$, one has
    \begin{align*}
    	\mathcal{g}_{\xi}(\eta,\eta')&= \frac{1}{\lambda_1}z_1z_1' + \frac{t}{\lambda_2}z_2z_2'  + \frac{1}{\lambda_1 + \lambda_2}\left(w_1 w_2'- w_2 w_1' \right) \\
    	&= \sum_{i=1}^2\frac{1}{\lambda_i'}z_iz_i'  + \frac{1}{\lambda_1' + t \lambda_2'}\left(w_1 w_2'- w_2 w_1' \right).
    \end{align*}
    \begin{rmk}
    	Recall that $\mathcal{g}$ is not uniquely determined by its value $\mathcal{g}_{\xi}$ at $\xi$. Let $X_{\alpha}$ denote the vector field over $D(t)^*$ generated by $L_{\alpha}$, $\alpha \in D(t)$.  We have
    	\begin{align*}
    		X_{e_i} = e_i \frac{\partial}{\partial e_i} + \frac{1}{2}&\left(x \frac{\partial}{\partial x} + y \frac{\partial}{\partial y}\right), \\
    		X_x = \frac{x}{2}\left(\frac{\partial }{\partial e_1} + \frac{\partial }{\partial e_2}\right) + (e_1 + te_2) \frac{\partial }{\partial y}, 
    		&\quad 
    		X_y = \frac{y}{2}\left(\frac{\partial }{\partial e_1} + \frac{\partial }{\partial e_2}\right) - (e_1 + te_2) \frac{\partial }{\partial x},
    	\end{align*}  
    	and
    	\begin{align*}
    		\mathcal{g}(X_{e_i}, X_{e_i}) = \frac{1}{e_i}, \quad \mathcal{g}(X_x,X_y) = e_1 + t e_2.
    	\end{align*}
    	Some straightforward computations show that
    	\begin{align*}
    		\mathcal{g}(\frac{\partial }{\partial e_i}, \frac{\partial }{\partial e_i}) = \frac{1}{e_i} + \frac{xy}{2e_i^2(e_1 + te_2)}, \quad \mathcal{g}(\frac{\partial }{\partial e_1}, \frac{\partial }{\partial e_2}) = \frac{xy}{2e_1e_2(e_1 + te_2)}. 
    	\end{align*}
    	The terms generated by the odd coordinates $x$ and $y$ cannot be detected by $\mathcal{g}_{\xi}$.
    \end{rmk}
    \begin{rmk}
    	Let $S_t\mathbb{R}^d:= \bigoplus_{i=1}^d D(t)$. $S_t\mathbb{R}^d$ is an one-parametric family of $(2d|2d)$-dimensional Jordan superalgebras. Let $\{e_i^A, x_i^A, e_i^B, x_i^B\}_{i=1}^d$ be a basis of  $S_t\mathbb{R}^d$ such that $e_i^A, e_i^B, x_i^A, x_i^B$ form a basis of $D(t)$ as in Example \ref{dt}.
    	For $\xi=\sum_{\alpha \in \{A,B\}} \sum_{i=1}^d \lambda_i^{\alpha} e_i^{\alpha*} $, $\eta = \sum_{\alpha \in \{A,B\}} \sum_{i=1}^d (z_i^{\alpha} e_i^{\alpha*}  + w_i^{\alpha} x_i^{\alpha*})$, and $\eta' = \sum_{\alpha \in \{A,B\}} \sum_{i=1}^d (z_i'^{\alpha} e_i^{\alpha*} + w_i'^{\alpha} x_i^{\alpha*})$, one has
    	\begin{align*}
    		\mathcal{g}_{\xi}(\eta,\eta')= \sum_{\alpha \in \{A,B\}} \sum_{i=1}^d \frac{1}{\lambda_i^{\alpha}} z_i^{\alpha}z_i'^{\alpha}  + \sum_{i=1}^d  \frac{1}{(\lambda_i^A+ t \lambda_i^B)}\left(w_i^A w_i'^B- w_i^B w_i'^A \right).
    	\end{align*}
    \end{rmk}
    
\section*{Acknowledgements}

F. M. C. acknowledges that this work has been supported by the Madrid Government (Comunidad de Madrid-Spain) under the Multiannual Agreement with UC3M in the line of “Research Funds for Beatriz Galindo Fellowships” (C\&QIG-BG-CM-UC3M), and in the context of the V PRICIT (Regional Program of Research and Technological Innovation).
L. S. acknowledges partial support by grant SCHW893/5-1 of the
Deutsche Forschungsgemeinschaft, and also expresses his gratitude for the hospitality of the Max Planck Institute for the Mathematics in the Sciences in Leipzig during numerous visits.
This publication is based upon work from COST Action CaLISTA CA21109 supported by COST (European Cooperation in Science and Technology, www.cost.eu).


\begin{bibsection}
		\begin{biblist}
			\bib{A-N-2000}{book}{
				author    = {Amari, S. I.},
				author	  ={Nagaoka, H.},
				publisher = {American Mathematical Society, Providence, RI},
				title     = {Methods of Information Geometry},
				date      = {2000},
			}
			\bib{Barbier2018}{article}{
				title={On structure and TKK algebras for Jordan superalgebras},
				author={Barbier, Sigiswald},
				author={Coulembier, Kevin},
				journal={Communications in Algebra},
				volume={46},
				number={2},
				pages={684--704},
				date={2018},
				publisher={Taylor \& Francis}
			}
			\bib{Carmeli11}{book}{
				title={Mathematical Foundations of Supersymmetry},
				author={Carmeli, C.},
				author={Caston, L.},
				author={Fioresi, R.},
				date={2011},
				volume={15},
				publisher={European Mathematical Society}
			}
			\bib{C-J-S-2020-02}{article}{
				author   = {Ciaglia, F. M.},
				author   = {Jost, J.},
				author   = {Schwachh\"{o}fer, L.},
				journal  = {Entropy},
				title    = {Differential geometric aspects of parametric estimation theory for states on finite-dimensional C*-algebras},
				date     = {2020},
				number   = {11},
				pages    = {1332},
				volume   = {22}
			}
		    \bib{Ciaglia23}{article}{
				title={Information geometry, Jordan algebras, and a coadjoint orbit-like construction}, 
				author={Ciaglia, Florio M. },
				author={Jost, Jürgen},
				author={Schwachhöfer, Lorenz},
				date={2023},
				journal={Symmetry, Integrability and Geometry: Methods and Applications (SIGMA)},
				volume={19},
				pages={078-27}
			}
			\bib{Kac77}{article}{
				title={Classification of simple $\mathbb{Z}$-graded Lie superalgebras and simple Jordan superalgebras.},
			    author={Kac, V.G.},
			    date={1977},
			    volume={5},
			    number={13},
			    pages={1375-1400}
			}
			\bib{Kirillov1962}{article}{
				title={Unitary representations of nilpotent Lie groups},
				author={Kirillov, Alexandre A.},
				journal={Russian Mathematical Surveys},
				volume={17},
				number={4},
				pages={53--104},
				date={1962}
			}
			\bib{Kirillov2012}{book}{
				title={Elements of the Theory of Representations},
				author={Kirillov, Alexandre A.},
				volume={220},
				date={2012},
				publisher={Springer Science \& Business Media}
			}
			\bib{Koecher99}{book}{
			    title={The Minnesota notes on Jordan algebras and their applications},
			    author={Koecher, M.},
			    date={1999},
			    volume={17},
			    publisher={Springer Science and Business Media}
			}
			\bib{Leites80}{article}{
				title={Introduction to the theory of supermanifolds},
			    author={Leites, D.A.},
			    date={1980},
			    journal={Russian Mathematical Surveys},
			    volume={35},
			    number={1},
			    pages={1}
			}
			\bib{Martin17}{article}{
				title={Classification of three-dimensional Jordan superalgebras},
			    author={Martin, M. E.},
			    date={2017},
			    eprint={1708.01963}
			}
			\bib{Paris-2009}{article}{
				author   = {Paris, M. G. A.},
				journal  = {International Journal of Quantum Information},
				title    = {Quantum Estimation for Quantum Technology},
				date     = {2009},
				number   = {1},
				pages    = {125--137},
				volume   = {7}
				}
			\bib{Racine2003}{article}{
				title={Simple Jordan superalgebras with semisimple even part},
				author={Racine, M. L.},
				author={Zel’manov, E. I.},
				journal={Journal of Algebra},
				volume={270},
				number={2},
				pages={374--444},
				date={2003},
				publisher={Elsevier}
			}
			\bib{Salmasian2010}{article}{
				title={Unitary representations of nilpotent super Lie groups},
				author={Salmasian, Hadi},
				journal={Communications in Mathematical Physics},
				volume={297},
				number={1},
				pages={189--227},
				date={2010},
				publisher={Springer}
			}
			\bib{Sussmann-1973}{article}{
				author   = {Sussmann, H. J.},
				journal  = {Transactions of the American Mathematical Society},
				title    = {Orbits of families of vector fields and integrability of distributions},
				date     = {1973},
				pages    = {171--188},
				volume   = {180},
}
		\end{biblist}
	\end{bibsection}
\end{document}